\documentclass[11pt, twoside]{article}
\usepackage{amsfonts}
\usepackage{amsmath,amssymb}
\usepackage{amsthm}
\usepackage[mathscr]{euscript}
\usepackage{mathrsfs}
\usepackage{indentfirst}
\usepackage{color}
\usepackage{enumitem}
\usepackage{accents}
\usepackage{lineno}

\newcommand{\lbl}[1]{\label{#1}}

\allowdisplaybreaks
\newtheorem{theo}{Theorem}[section]

\newtheorem{defi}{Definition}[section]
\newcommand{\be}{\begin{equation}}
\newcommand{\ee}{\end{equation}}
\newcommand\bes{\begin{eqnarray}} \newcommand\ees{\end{eqnarray}}
\newcommand{\bess}{\begin{eqnarray*}}
\newcommand{\eess}{\end{eqnarray*}}

\newcommand\ep{\varepsilon}
\newcommand\kk{\left}
\newcommand\rr{\right}
\newcommand\dd{\displaystyle}
\newcommand\vp{\varphi}

\newcommand\lm{\lambda}

\newcommand\yy{\infty}

\newcommand\R{\mathbb{R}}
\newcommand\ol{\overline}
\newcommand\ud{\underline}

\newcommand\oo{\Omega}\newcommand\pt{\partial}

\newcommand\bds{\boldsymbol}
\newcommand\qq{\eqref}

\nolinenumbers

\begin{document}\thispagestyle{empty}

\begin{center}{\Large\bf A vector-host epidemic model with spatial structure and seasonality}\\[4mm]
Mingxin Wang$^\dag$,\;\;\; Qianying Zhang$^{\ddag,}\,$\footnote{Corresponding author. {\sl E-mail}: zhangqianying@tiangong.edu.cn}

$^\dag${\small School of Mathematics and Information Science, Henan Polytechnic University, Jiaozuo 454000, China}

$^\ddag${\small School of Mathematical Sciences, Tiangong University, Tianjin 300387, China}\end{center}

\begin{quote}
\noindent{\bf Abstract.} Recently, Li and Zhao \cite{LZ21} (Bull. Math. Biol., 83(5), 43, 25 pp (2021))  proposed and studied a periodic reaction-diffusion model of Zika virus with seasonality and spatial heterogeneous structure in host and vector populations. They found the basic reproduction ratio $R_0$, which is a threshold parameter. In this short paper we shall use the upper and lower solutions method to study the model of \cite{LZ21} with Neumann boundary conditions replaced by general boundary conditions. Our approach can greatly simplify the analysis by Magal et al. \cite{MWW18} and Li \& Zhao \cite{LZ21}.


\end{quote}

\pagestyle{myheadings}
\section{Introduction}{\setlength\arraycolsep{2pt}
\markboth{\rm$~$ \hfill A vector-host epidemic model\hfill $~$}{\rm$~$ \hfill M.X. Wang \& Q.Y. Zhang\hfill $~$}

Zika virus is a mosquito-borne flavivirus that is primarily transmitted to humans through the bites of the Aedes aegypti and Aedes albopictus. Zika virus was first isolated from a rhesus monkey in the Zika Forest of Uganda in 1947 (\cite{DKH}). The first human outbreak of Zika virus has occurred in Island of Yap in 2007. A large outbreak in Brazil was occurred and provided a large number of infected cases in May 2015. Since then, it has spread rapidly to many other countries.

Mathematical modeling has become an important tool used to describe the spread of Zika virus. Fitzgibbon et al. \cite{FMW17} studied the well-posedness of the vector-host epidemic model with spatial heterogeneous structure and the numerical simulations of a periodic model with a time dependent vector breeding rate, respectively. Magal et al. \cite{MWW18} provided a detailed analysis of the reaction-diffusion model proposed in \cite{FMW17} and proved that the basic reproduction ratio $R_0$ serves as a threshold value to established the evolution dynamics of the model. In order to study the impact of spatiotemporal heterogeneities and movements on the spread and persistence of diseases, it is essential to investigate the role of diffusion and seasonality in the transmission of diseases in a heterogeneous environment.

Divide the population into two subpopulations: the host and vector populations, and suppose that all populations are living in a bounded domain $\Omega\subset\R^n$ with smooth boundary $\partial\Omega$.  Let $H_i(x,t)$, $V_u(x,t)$ and $V_i(x,t)$ be the densities of infected hosts, susceptible vectors, and infected vectors at location $x$ and time $t$, respectively. Recently, Li and Zhao \cite{LZ21} proposed and studied the following periodic reaction-diffusion model of Zika virus with seasonality and spatial heterogeneous structure in host and vector populations:\vspace{-1mm}
\bes\left\{\begin{array}{ll}
\partial_t H_i-\nabla\cdot d_1(x,t)\nabla H_i=-\rho(x,t)H_i+\sigma_1(x,t)H_u(x,t)V_i,\;\;&x\in\Omega,\;t>0,\\[1mm]
\partial_t V_u-\nabla\cdot d_2(x,t)\nabla V_u=-\sigma_2(x,t)V_uH_i+\beta(x,t)(V_u+V_i)\\[1mm]
\hspace{43mm}-\mu_1(x,t) V_u-\mu_2(x,t)(V_u+V_i)V_u,\;\;&x\in\Omega,\;t>0,\\[1mm]
\partial_t V_i-\nabla\cdot d_2(x,t)\nabla V_i=\sigma_2(x,t)V_uH_i-\mu_1(x,t) V_i-\mu_2(x,t)(V_u+V_i)V_i,\;\;&x\in\Omega,\;t>0,\\[1mm]
(H_i(\cdot, 0), V_u(\cdot, 0), V_i(\cdot, 0))=(H_{i0}, V_{u0}, V_{i0})\in C(\overline\Omega;\,\mathbb{R}^3_+)
	\end{array}\right.\qquad\label{1.1}\vspace{-1mm}
\ees
with the homogeneous Neumann boundary conditions
  \bess
  \frac{\partial H_i}{\partial\nu}=\frac{\partial V_u}{\partial\nu}=\frac{\partial V_i}{\partial\nu}=0,\;\;\;x\in\pt\oo,\; t>0,
  \lbl{1.2a}\eess
where $\nu$ is the outward normal vector of $\partial\Omega$. The coefficient  functions satisfy  \vspace{-2mm}
\begin{enumerate}[leftmargin=10mm]
\item[{\bf(H)}]\, All coefficient functions are $T$-periodic in $t$, H\"{o}lder continuous, nonnegative and nontrivial; $\sigma_1(x,t)H_u(x,t)\not\equiv 0$ for $(x,t)\in\Omega\times(0,T]$; $\rho(x,t)$, $\sigma_2(x,t)$, $\mu_1(x,t)$, $d_1(x,t)$ and $d_2(x,t)$ are positive, and $\nabla_xd_i$ is H\"{o}lder continuous for $(x,t)\in\overline\Omega\times[0,T]$, $i=1,2$.\vspace{-2mm}
\end{enumerate}
Li and Zhao \cite{LZ21} introduced the basic reproduction ratio $R_0$ and shown  that the disease-free periodic solution is globally asymptotically stable if $R_0\le 1$, while the positive periodic solution is globally asymptotically stable if $R_0>1$.

In this short paper we use the upper and lower solutions method to study the dynamics of the model \qq{1.1} with general boundary conditions:
 \bes
 a_1\frac{\partial H_i}{\partial\nu}+b_1(x,t)H_i
 =a_2\frac{\partial V_u}{\partial\nu}+b_2(x,t)V_u
 =a_2\frac{\partial V_i}{\partial\nu}+b_2(x,t)V_i=0,\;\;\;x\in\pt\oo,\;t>0,
 \lbl{1.2}\ees
where either (i) $a_k=0$, $b_k=1$; or (ii) $a_k=1$, $b_k(x,t)\geq 0$ with $b_k\in C^{1+\alpha,\, (1+\alpha)/2}(\pt\oo\times[0,T])$ and $T$-periodic in $t$ for $k=1,2$. This approach can greatly simplify the analysis by Magal et al. \cite{MWW18} and Li \& Zhao \cite{LZ21}.

\section{Eigenvalue problems}\setcounter{equation}{0} {\setlength\arraycolsep{2pt}

Let
 $$\mathscr{L}_k=\partial_t-a^{ij}_k(x,t)D_{ij}+b^i_k(x,t)D_i,\;\;1\leq k\leq m$$
be strongly parabolic operator in $\oo\times(0,T]$, $a^{ij}_k,b^i_k\in C^{\alpha,\alpha/2}(\overline\Omega\times[0,T])$, $i,j=1,\cdots,n$ and
 $$ \mathscr{B}_k=a_k\partial_{\nu}+b_k(x,t), \ \ 1\leq k\leq m,$$
where either (i) $a_k=0$, $b_k=1$; or (ii) $a_k=1$, $b_k(x,t)\geq 0$ with $b_k\in C^{1+\alpha,\, (1+\alpha)/2}(\pt\oo\times[0,T])$. Define
 \[\mathscr{L}={\rm diag}(\mathscr{L}_1,\ldots,\mathscr{L}_m),\;\;\;
  \mathscr{B}={\rm diag}(\mathscr{B}_1,\ldots,\mathscr{B}_m).\]
Let $H(x,t)=(h_{ij}(x,t))_{m\times m}$ and $h_{ij}\in C^{\alpha,\alpha/2}(\overline\Omega\times[0,T])$, $u=(u_1,\cdots,u_m)^T$.

\begin{defi}\lbl{def2.1} Let $X$ be a Banach space composed of functions defined in $\overline\Omega$. We say that the operator $(\mathscr{L}, H, \mathscr{B})$ is {\it regular} if for any $\varphi\in X$, the linear problem
 \bes\begin{cases}
 \mathscr{L}u=H(x,t)u,\;\;\;&x\in\Omega,\;0<t\le T,\\
 \mathscr{B}[u]=0,\;\;&x\in\partial\Omega,\;0<t\le T,\\
 u(x,0)=\varphi(x),&x\in\Omega
 \end{cases}\lbl{2.1}\ees
has a unique solution $u(x,t)$ and the Poincar\'{e} map ${\mathcal A}$ defined by ${\mathcal A}\varphi=u(x,T)$ is compact in $X$.

Let $X$ be a Banach space, $P$ be a solid positive cone in $X$ and $P^\circ$ be the interior of $P$. We say that $(\mathscr{L}, H, \mathscr{B})$ has the {\it strong maximum principle property} if, for any $\varphi\in P\setminus\{0\}$, the unique solution $u(x,t)$ of \qq{2.1} satisfies $u(x,t)\in P^\circ$ for all $0<t\le T$.
\end{defi}

We call that $H$ is cooperative (essentially positive) if $h_{ij}(x,t)\ge 0$ for all $i\not=j$ and $(x,t)\in\Omega\times\R$, and that $H$ is fully coupled if the index set $\{1,\cdots, m\}$ cannot be split up in two disjoint nonempty  sets ${\cal I}$ and ${\cal J}$ such that $h_{ij}(x,t)\equiv 0$ in $\Omega\times\R$  for $i\in{\cal I}, j\in{\cal J}$.

We mention that if $H$ is cooperative and fully coupled, then $(\mathscr{L}, H, \mathscr{B})$ has the strong maximum principle property by the strong maximum principle for parabolic equations.

\begin{theo}\lbl{th2.1} Assume that $a^{ij}_k, b^i_k, b_k$ and $h_{ij}$ are  time-periodic with period $T$. If the operator $(\mathscr{L}, H, \mathscr{B})$ is regular and has the  strong maximum principle property, then the time-periodic parabolic eigenvalue problem
 \bes\begin{cases}
 \mathscr{L}\phi=H(x,t)\phi+\lm\phi,\;\;&x\in\Omega,\;0<t\le T,\\
 \mathscr{B}[\phi]=0,\;\;&x\in\partial\Omega,\;0<t\le T,\\
 \phi(x,0)=\phi(x,T),\;\;&x\in\Omega
 \end{cases}\lbl{2.2}\ees
has a unique principal eigenvalue $\lm(\mathscr{L}, H, \mathscr{B}; T)$ with positive eigenfunction $\phi$, i.e., $\phi_i(x,t)>0$ for $(x,t)\in\Omega\times(0,T]$, $i=1,\cdots, m$.

Especially, when $a^{ij}_k, b^i_k, b_k$ and $h_{ij}$ do not depend on $t$, then $\lm(\mathscr{L}, H, \mathscr{B}; 1)$ is the unique principal eigenvalue of \qq{2.2} with positive eigenfunction $\tilde\phi(x)=\int_0^1\phi(x,t){\rm d}t$.
\end{theo}

\begin{proof}\hspace{-1mm}\footnote{The idea here originated from a discussion between the first author and Professor Xing Liang of the University of Science and Technology of China. The authors expresses their deep gratitude to Professor Xing Liang.}  Since the operator $(\mathscr{L}, H, \mathscr{B})$ is regular and has the strong maximum principle property,  for any given $\varphi\in P\setminus\{0\}$, the Poincar\'{e} map ${\mathcal A}$ defined by ${\mathcal A}\varphi=u(x,T)$ is compact in $X$ and strongly positive with respect to $P$. In view of the Krein-Rutman theorem (the strong version), $r({\mathcal A})$, the spectral radius of ${\mathcal A}$, is positive and is the unique principal eigenvalue of $r({\mathcal A})$ with positive eigenfunction $\varphi\in P^\circ$, i.e.,
 \[{\mathcal A}\varphi=r({\mathcal A})\varphi,\;\;\;{\rm or}\;\;u(x,T)=r({\mathcal A})u(x,0),\;\;\;u(x,0)\in P^\circ.\]
Set
 \[\lm=-\frac{\ln r({\mathcal A})}T,\;\;\;\phi(x,t)={\rm e}^{\lm t}u(x,t).\]
It is easy to verify that $(\lm, \phi)$ satisfies \qq{2.2}, i.e., $\lm$ is the unique principal eigenvalue of \qq{2.2} with positive eigenfunction $\phi(x,t)$.

It is obvious that when $a^{ij}_k, b^i_k, b_k$ and $h_{ij}$ do not depend on $t$, then $\lm(\mathscr{L}, H, \mathscr{B}; 1)$ is the unique principal eigenvalue of \qq{2.2} with positive eigenfunction $\tilde\phi(x)=\int_0^1\phi(x,t){\rm d}t$.
\end{proof}

\section{Positive time periodic solutions}\setcounter{equation}{0} {\setlength\arraycolsep{2pt}

In this section, we consider the following time periodic problem associated to \qq{1.1} and \qq{1.2}:
\bes\begin{cases}
\partial_t\bds{H}_{\!i}-\nabla\cdot d_1(x,t)\nabla \bds{H}_{\!i}=-\rho(x,t)\bds{H}_{\!i}+\sigma_1(x,t) H_u(x,t)\bds{V}_{\!i},\;\;&x\in\Omega,\;0<t\le T,\\
\partial_t\bds{V}_{\!u}-\nabla\cdot d_2(x,t)\nabla \bds{V}_{\!u}=-\sigma_2(x,t){\bds{V}_{\!u}}\bds{H}_{\!i} +\beta(x,t)({\bds{V}_{\!u}}+\bds{V}_{\!i})\\  \hspace{45mm}-\mu_1(x,t){\bds{V}_{\!u}}-\mu_2(x,t)({\bds{V}_{\!u}}
+\bds{V}_{\!i}){\bds{V}_{\!u}},\;\;&x\in\Omega,\;0<t\le T,\\
\partial_t\bds{V}_{\!i}-\nabla\cdot d_2(x,t)\nabla\bds{V}_{\!i}
=\sigma_2(x,t){\bds{V}_{\!u}}\bds{H}_{\!i}-\mu_1(x,t) \bds{V}_{\!i}-\mu_2(x,t)({\bds{V}_{\!u}}
+\bds{V}_{\!i})\bds{V}_{\!i},\;\;&x\in\Omega,\;0<t\le T,\\[2mm]
a_1\dd\frac{\partial\bds{H}_{\!i}}{\partial\nu}+b_1(x,t)\bds{H}_{\!i}
=a_2\dd\frac{\partial\bds{V}_{\!u}}{\partial\nu}+b_2(x,t)\bds{V}_{\!u}
=a_2\dd\frac{\partial\bds{V}_{\!i}}{\partial\nu}+b_2(x,t)\bds{V}_{\!i}=0,\;\;&x\in\partial\Omega,\;0<t\le T,\\[1mm]
\bds{H}_{\!i}(x,0)=\bds{H}_{\!i}(x,T), \;\; {\bds{V}_{\!u}}(x,0)={\bds{V}_{\!u}}(x,T), \;\; \bds{V}_{\!i}(x,0)=\bds{V}_{\!i}(x,T),\;\;&x\in\Omega.
\end{cases}\qquad\label{3.1}\vspace{-1mm}
 \ees
Assume that $(\bds{H}_{\!i}, \bds{V}_{\!u}, \bds{V}_{\!i})$ is a nonnegative solution of \eqref{3.1} and set $\bds{V}=\bds{V}_{\!u}+\bds{V}_{\!i}$. Then $\bds{V}$ satisfies
 \bes\left\{\begin{array}{lll}
\partial_t\bds{V}-\nabla\cdot d_2(x,t)\nabla\bds{V}=(\beta(x,t)-\mu_1(x,t))\bds{V}
-\mu_2(x,t)\bds{V}^2,\;\;\;&x\in\Omega,\;0<t\le T,\\[2mm]
a_2\dd\frac{\partial\bds{V}}{\partial\nu}+b_2(x,t)\bds{V}=0,\;\;&x\in\partial\Omega,\;0<t\le T,\\[1mm]
\bds{V}(x,0)=\bds{V}(x,T),\;\;\;\;&x\in\Omega.
 \end{array}\right.\label{3.3}\ees
Let $\zeta(\mu_1,\beta)$ be the principal eigenvalue of
 \bes\begin{cases}
\partial_t\vp-\nabla\cdot d_2(x,t)\nabla\varphi+(\mu_1(x,t)-\beta(x,t))\varphi=
 \zeta\varphi,\;\;&x\in\Omega,\;0<t\le T,\\[2mm]
a_2\dd\frac{\partial\varphi}{\partial\nu}+b_2(x,t)\varphi=0,\;\;&x\in\partial\Omega,\;0<t\le T,\\[1mm]
\varphi(x,0)=\varphi(x,T),\;\;\;\;&x\in\Omega.
	\end{cases}\lbl{3.2}\ees
Thanks to \cite[Theorem 28.1]{Hess}, \eqref{3.3} has no positive solution when $\zeta(\mu_1,\beta)\ge 0$. Therefore, ${\bds{V}}=0$, i.e., ${\bds{V}_{\!u}}=\bds{V}_{\!i}=0$, and then $\bds{H}_{\!i}=0$ by the first equation of \eqref{3.1}. If $\zeta(\mu_1,\beta)<0$, then \eqref{3.3} has a unique positive solution ${\bds{V}}$ and ${\bds{V}}$ is globally asymptotically stable in the periodic sense.

Assume that $\zeta(\mu_1,\beta)<0$ and let ${\bds{V}}$ be the unique positive solution of \eqref{3.3}. To investigate the positive solutions of \eqref{3.1} is equivalent to study the positive solutions $(\bds{H}_{\!i}, \bds{V}_{\!i})$ of\vspace{-1mm}
 \bes\left\{\begin{array}{lll}
\partial_t\bds{H}_{\!i}-\nabla\cdot d_1(x,t)\nabla \bds{H}_{\!i}=-\rho(x,t)\bds{H}_{\!i}+\sigma_1(x,t) H_u(x,t)\bds{V}_{\!i},\;\;&x\in\Omega,\;0<t\le T,\\[1mm]
\partial_t\bds{V}_{\!i}-\nabla\cdot d_2(x,t)\nabla\bds{V}_{\!i}=\sigma_2(x,t)({\bds{V}}(x,t)
-\bds{V}_{\!i})\bds{H}_{\!i}\\[1mm]
\hspace{43mm}
-\big(\mu_1(x,t)+\mu_2(x,t){\bds{V}}(x,t)\big)\bds{V}_{\!i},\;\;\;&x\in\Omega,\;0<t\le T,\\[2mm]
a_1\dd\frac{\partial\bds{H}_{\!i}}{\partial\nu}+b_1(x,t)\bds{H}_{\!i}=a_2\dd\frac{\partial\bds{V}_{\!i}}{\partial\nu}+b_2(x,t)\bds{V}_{\!i}=0,\;\;&x\in\partial\Omega,\;0<t\le T,\\[2mm]
\bds{H}_{\!i}(x,0)=\bds{H}_{\!i}(x,T),\;\; \bds{V}_{\!i}(x,0)=\bds{V}_{\!i}(x,T),\;\;\;\;&x\in\Omega
\end{array}\right.\label{3.4}\ees
satisfying $\bds{V}_{\!i}<\bds{V}$. The linearized eigenvalue problem of \eqref{3.4} at $(0,0)$ is
 \bes\left\{\begin{array}{lll}
\partial_t\phi_1-\nabla\cdot d_1(x,t)\nabla\phi_1+\rho(x,t)\phi_1
=\sigma_1(x,t)H_u(x,t)\phi_2+\lm\phi_1,\;\;&x\in\Omega,\;0<t\le T,\\[1mm]
\partial_t\phi_2-\nabla\cdot d_2(x,t)\nabla\phi_2+\big(\mu_1(x,t)+\mu_2(x,t){\bds{V}}(x,t)\big)\phi_2\\[1mm]
\hspace{57mm}
=\sigma_2(x,t){\bds{V}}(x,t)\phi_1+\lm\phi_2,\;\;\;&x\in\Omega,\;0<t\le T,\\[2mm]
a_1\dd\frac{\partial\phi_1}{\partial\nu}+b_1(x,t)\phi_1=a_2\dd\frac{\partial\phi_2}{\partial\nu}+b_2(x,t)\phi_2=0,\;\;&x\in\partial\Omega,\;0<t\le T,\\[2mm]
\phi_1(x,0)=\phi_1(x,T), \ \ \phi_2(x,0)=\phi_2(x,T),\;\;\;\;&x\in\Omega.
\end{array}\right.\label{3.5}\ees
By Theorem \ref{th2.1}, the problem \eqref{3.5} has a unique principal eigenvalue $\lm(\bds{V})$ with positive eigenfunction $\phi=(\phi_1,\phi_2)^T$.

\begin{theo}\lbl{th3.1} Assume that $\zeta(\mu_1,\beta)<0$. Then problem \eqref{3.4} has a positive solution $(\bds{H}_{\!i}, \bds{V}_{\!i})$ if and only if $\lm(\bds{V})<0$. Moreover, the positive solution $(\bds{H}_{\!i}, \bds{V}_{\!i})$ of \eqref{3.4} is unique and satisfies $\bds{V}_{\!i}<{\bds{V}}$ when it exists. Therefore, \eqref{3.1}  has a positive solution $(\bds{H}_{\!i}, \bds{V}_{\!u}, \bds{V}_{\!i})$ if and only if $\lm(\bds{V})<0$, and $(\bds{H}_{\!i}, \bds{V}_{\!u}, \bds{V}_{\!i})$ is unique and takes the form $(\bds{H}_{\!i}, \bds{V}-\bds{V}_{\!i}, \bds{V}_{\!i})$  when it exists.
\end{theo}

\begin{proof}  {\bf The necessity}. Define
 \bess
 \mathbb{X}=\big\{\psi=(\psi_1,\psi_2):\,\psi_k\in C^{2+\alpha}(\overline\Omega),\;a_k\frac{\partial \psi_k}{\partial\nu}+b_k(\cdot,0)\psi_k=0\;\;{\rm on}\;\;\partial\Omega,\;k=1,2\big\},\eess
and
 \bess
L(t)=-\left(\begin{array}{cc} \nabla\cdot d_1(\cdot,t)\nabla-\rho(\cdot,t)+\lm(\bds{V}) & \sigma_1(\cdot, t)H_u(\cdot,t)\\[2mm]
 \sigma_2(\cdot,t)\bds{V} (\cdot,t)\;\; &\;\; \nabla\cdot d_2(\cdot,t)\nabla-\big(\mu_1(\cdot,t)+\mu_2(\cdot,t){\bds{V}} (\cdot,t)\big)+\lm(\bds{V}) \end{array}\right).
 \eess
For any given $t\geq 0$, let $Q(t)$ be the semigroup generated by $L(t)$ in $\mathbb{X}$. Then the Poincar\'{e} map $Q(T):\,\mathbb{X}\to \mathbb{X}$  is strongly positive and compact. Since $\phi=(\phi_1,\phi_2)^T$ satisfies \qq{3.5} and is positive, we have $\phi(\cdot,0)=\phi(\cdot,T)=Q(T)\phi(\cdot,0)$. It then follows that $r(Q(T))=1$.

If \qq{3.4} has a positive solution $(\bds{H}_{\!i}, \bds{V}_{\!i})$, then $ \bds{\Phi}(\cdot,t)=(\bds{H}_{\!i}(\cdot,t), \bds{V}_{\!i}(\cdot,t))^T$ satisfies
  \bess\begin{cases}
\frac{{\rm d}{\bds{\Phi}}}{{\rm d}t}+L(t){\bds{\Phi}}=-F(t;\bds{\Phi}), \ \,t>0,\\
\bds{\Phi}(\cdot,0)=\bds{\Phi}(\cdot,T),
 \end{cases}\eess
where
 \bess
F(t;\bds{\Phi})=\left(\begin{array}{cc} \lm(\bds{V})\bds{H}_{\!i}(\cdot,t)\\[2mm] \sigma_2(\cdot,t)\bds{V}_{\!i}(\cdot,t)\bds{H}_{\!i}(\cdot,t)+\lm(\bds{V})\bds{V}_{\!i}(\cdot,t) \end{array}\right).
 \eess
Therefore,
 \bess
r(Q(T))\bds{\Phi}(\cdot,0)=\bds{\Phi}(\cdot,0)=\bds{\Phi}(\cdot,T)=Q(T)\bds{\Phi}(\cdot,0)-\int_0^T Q(T-s)F(s;\bds{\Phi}(\cdot,s)){\rm d}s.
\eess
We assume on the contradiction that $\lm(\bds{V})\geq 0$. Then
  \bess
F(t;\bds{\Phi})\ge
\left(\begin{array}{cc} 0\\[1mm] \sigma_2(\cdot,t)\bds{V}_{\!i}(\cdot,t)\bds{H}_{\!i}(\cdot,t) \end{array}\right),
 \eess
and $\sigma_2(\cdot,t)\bds{V}_{\!i}(\cdot,t)\bds{H}_{\!i}(\cdot,t)>0$ in $[0,T]$. This is impossible by the conclusion \cite[Theorem 3.2 (iv)]{Am76}.

{\bf The sufficiency}.  Assume $\lm(\bds{V})<0$, we shall show that \eqref{3.4} has a unique positive solution $(\bds{H}_{\!i}, \bds{V}_{\!i})$ and $\bds{V}_{\!i}<{\bds{V}}$. To this aim, we prove the following general conclusion.
\begin{enumerate}[leftmargin=6mm]
\item[$\bullet$] Let $\varphi$ be the positive eigenfunction corresponding to $\zeta(\mu_1,\beta)$ of problem \eqref{3.2}. Normalize $\varphi$ by $\|\varphi\|_{L^{\infty}(\oo\times(0,T])}=1$. Then there is $0<\ep_0\ll 1$ such that, when $|\ep|\le\ep_0$, problems\vspace{-1mm}
 \bes\begin{cases}
\partial_t\bds{H}_{\!i}\!-\!\nabla\cdot d_1(x,t)\nabla\bds{H}_{\!i}\!=-\rho(x,t)\bds{H}_{\!i}+\sigma_1(x,t)H_u(x,t) \bds{V}_{\!i},&x\in\Omega,\;0<t\le T,\\
\partial_t\bds{V}_{\!i}\!-\!\nabla\cdot d_2(x,t)\nabla\bds{V}_{\!i}=\sigma_2(x,t)({\bds{V}}(x,t)
+\ep\varphi(x,t)-\bds{V}_{\!i})^+ \bds{H}_{\!i}\\
 \hspace{41mm}-\big(\mu_1(x,t)+\mu_2(x,t)({\bds{V}}(x,t)\!-\!\ep\varphi(x,t))\big)\bds{V}_{\!i},&x\in\Omega,\;0<t\le T,\\[2mm]
a_1\dd\frac{\partial\bds{H}_{\!i}}{\partial\nu}+b_1(x,t)\bds{H}_{\!i}=a_2\dd\frac{\partial\bds{V}_{\!i}}{\partial\nu}+b_2(x,t)\bds{V}_{\!i}=0, &x\in\partial\Omega,\;0<t\le T,\\[1mm]
 \bds{H}_{\!i}(x, 0)=\bds{H}_{\!i}(x, T),\;\;\; \bds{V}_{\!i}(x, 0)=\bds{V}_{\!i}(x, T), &x\in\Omega \vspace{-1mm}
	\end{cases}\lbl{3.6}\ees
 and
\bes\begin{cases}
\partial_t\bds{H}_{\!i}\!-\!\nabla\cdot d_1(x,t)\nabla\bds{H}_{\!i}\!=-\rho(x,t)\bds{H}_{\!i}+\sigma_1(x,t)H_u(x,t) \bds{V}_{\!i},&x\in\Omega,\;0<t\le T,\\
\partial_t\bds{V}_{\!i}\!-\!\nabla\cdot d_2(x,t)\nabla\bds{V}_{\!i}=\sigma_2(x,t)({\bds{V}}(x,t)
+\ep\varphi(x,t)-\bds{V}_{\!i})\bds{H}_{\!i}\\
 \hspace{41mm}-\big(\mu_1(x,t)+\mu_2(x,t)({\bds{V}}(x,t)
 \!-\!\ep\varphi(x,t))\big)\bds{V}_{\!i},&x\in\Omega,\;0<t\le T,\\[2mm]
a_1\dd\frac{\partial\bds{H}_{\!i}}{\partial\nu}+b_1(x,t)\bds{H}_{\!i}=a_2\dd\frac{\partial\bds{V}_{\!i}}{\partial\nu}+b_2(x,t)\bds{V}_{\!i}=0, &x\in\partial\Omega,\;0<t\le T,\\[1mm]
 \bds{H}_{\!i}(x, 0)=\bds{H}_{\!i}(x, T),\;\;\; \bds{V}_{\!i}(x, 0)=\bds{V}_{\!i}(x, T),&x\in\Omega \vspace{-1mm}
	\end{cases}\lbl{3.6a}\ees
have unique positive solutions $(\bds{H}_{\!i,\ep}^+, \bds{V}_{\!i, \ep}^+)$ and $(\bds{H}_{\!i,\ep}, \bds{V}_{\!i, \ep})$, respectively. Moreover,  $\bds{V}_{\!i, \ep}^+, \bds{V}_{\!i, \ep}<\bds{V}+\ep\varphi$ for $(x,t)\in\Omega\times[0,T]$, which implies that $(\bds{H}_{\!i,\ep}^+, \bds{V}_{\!i, \ep}^+)=(\bds{H}_{\!i,\ep}, \bds{V}_{\!i, \ep})$, and \qq{3.6} and \qq{3.6a} are equivalent.\vspace{-1mm}
\end{enumerate}

Here we only discuss the problem \eqref{3.6} since problem \eqref{3.6a} can be dealt with by the same way.

{\it Existence of positive solutions of \eqref{3.6}}.
Let $\lm({\bds{V}}; \ep)$ be the principal eigenvalue of
 \bes\begin{cases}
\partial_t\phi_1-\nabla\cdot d_1(x,t)\nabla\phi_1+\rho(x,t)\phi_1-\sigma_1(x,t)H_u(x,t)\phi_2
 =\lm\phi_1,\;\;&x\in\Omega,\;0<t\le T,\\
\partial_t\phi_2-\nabla\cdot d_2(x,t)\nabla\phi_2-\sigma_2(x,t)\big({\bds{V}}(x,t)\!+\!\ep\varphi(x,t)\big)\phi_1 \\ \hspace{20mm}+\big(\mu_1(x,t)+\mu_2(x,t)({\bds{V}}(x,t)\!-\!\ep\varphi(x,t))\big)\phi_2
=\lm\phi_2,\;\;&x\in\Omega,\;0<t\le T,\\[2mm]
 a_1\dd\frac{\partial\phi_1}{\partial\nu}+b_1(x,t)\phi_1=a_2\dd\frac{\partial\phi_2}{\partial\nu}+b_2(x,t)\phi_2=0,\;\;&x\in\partial\Omega,\;0<t\le T,\\[1mm]
 \phi_1(x,0)=\phi_1(x,T), \;\; \phi_2(x,0)=\phi_2(x,T),\;\;&x\in\Omega.
	\end{cases}\lbl{3.7}\ees
Since $\lm({\bds{V}})<0$, there exists a $0<\ep_0\ll 1$ such that, for any $|\ep|\le\ep_0$ we have $\lm({\bds{V}}; \ep)<0$ by the continuity of $\lm({\bds{V}}; \ep)$ in $\ep$, and
  \bes
{\bds{V}}\pm\ep\varphi>0, \;\; (\ep\varphi)^2\mu_2-\ep\varphi|\beta+\zeta(\mu_1,\beta)|<\beta{\bds{V}},\;\;\;
(x,t)\in\Omega\times[0,T]\lbl{3.8}\ees
(when $a_2=0$ and $b_2=1$, the above inequalities can be proved by adapting analogous methods as in \cite[Lemma 3.7]{Wpara}).

Now we consider the linear problem
 \bes\begin{cases}
\partial_t\ol{\bds{H}}_{\!i}-\nabla\cdot d_1(x,t)\nabla\ol{\bds{H}}_{\!i}+\rho(x,t)\ol{\bds{H}}_{\!i}\\
\hspace{15mm}=\sigma_1(x,t)H_u(x,t)({\bds{V}}(x,t)+\ep\varphi(x,t)),\;\;&x\in\Omega,\;0<t\le T,\\[2mm]
a_1\dd\frac{\partial\ol{\bds{H}}_{\!i}}{\partial\nu}+b_1(x,t)\ol{\bds{H}}_{\!i}=0,\;\;&x\in\partial\Omega,\;0<t\le T,\\[1mm]
 \ol{\bds{H}}_{\!i}(x,0)=\ol{\bds{H}}_{\!i}(x,T),\;\;&x\in\Omega.
	\end{cases}\lbl{3.9} \ees
Let $\gamma(\rho)$ be the principal eigenvalue of
 \bess\begin{cases}
\partial_t\eta-\nabla\cdot d_1(x,t)\nabla\eta+\rho(x,t)\eta=\gamma\eta,\;\;&x\in\Omega,\;0<t\le T,\\[2mm]
a_1\dd\frac{\partial\eta}{\partial\nu}+b_1(x,t)\eta=0,\;\;&x\in\partial\Omega,\;0<t\le T,\\[1mm]
 \eta(x,0)=\eta(x,T),\;\;&x\in\Omega.
	\end{cases} \eess
Then $\gamma(\rho)>0$ since $\rho(x,t)>0$. Let $\eta(x,t)$ be the positive eigenfunction corresponding to $\gamma(\rho)$ with $\|\eta\|_{L^\infty(\oo\times(0,T])}=1$. We can find a positive constant $C$ such that
 \bess
 \sigma_1(x,t)H_u(x,t)({\bds{V}}(x,t)+\ep\varphi(x,t))\le C\gamma(\rho)\eta(x,t)\eess
for $(x,t)\in\Omega\times(0,T]$ (when $a_1=0$ and $b_1=1$, proofs of the above inequalities are similar to that of \cite[Lemma 3.7]{Wpara}). Clearly, $C\eta(x,t)$ and $0$ are the ordered upper and lower solutions of \qq{3.9}. By the upper and lower solutions method (\cite[Theorem 7.3]{Wpara}), the problem \qq{3.9} has at least one solution, denoted by $\ol{\bds{H}}_{\!i}$,  which is positive since
 \[\sigma_1(x,t)H_u(x,t)({\bds{V}}(x,t)+\ep\varphi(x,t))\ge 0,\,\not\equiv 0.\]
Clearly, $\ol{\bds{H}}_{\!i}$ is unique.

Making use of \qq{3.8} and \qq{3.9} we can show that $(\ol{\bds{H}}_{\!i}(x,t), {\bds{V}}(x,t)+\ep\varphi(x,t))$ is a strict upper solution of \eqref{3.6}. Let $(\phi_1^{\ep},\phi_2^{\ep})$ be the positive eigenfunction corresponding to $\lm({\bds{V}}; \ep)$. It is easy to verify that $\delta(\phi_1^{\ep}, \phi_2^{\ep})$ is a lower solution of \eqref{3.6}, and \bess
 \delta(\phi_1^{\ep}(x,t), \phi_2^{\ep}(x,t))\le(\ol{\bds{H}}_{\!i}(x,t), {\bds{V}}(x,t)+\ep\varphi(x,t)),\;\;\;(x,t)\in\Omega\times[0,T]
 \eess
provided that $\delta>0$ is suitably small. By the upper and lower solutions method (\cite[Theorem 7.15]{Wpara}), \eqref{3.6} has at least one positive solution $(\bds{H}^+_{\!i,\ep}, \bds{V}^+_{\!i, \ep})$, and $\bds{V}^+_{\!i, \ep}<\bds{V}+\ep\varphi$ for $(x,t)\in\Omega\times[0,T]$.\vskip 2pt

{\it Uniqueness of positive solutions of \eqref{3.6}}. Let $({\bds H}^*_{\!i, \ep}, {\bds V}^*_{\!i,\ep})$ be another positive periodic solution of \eqref{3.6}. We can find a constant $0<s<1$ such that $s({\bds H}^*_{\!i,\ep}, {\bds V}^*_{\!i,\ep})\le(\bds{H}^+_{\!i,\ep}, \bds{V}^+_{\!i,\ep})$ for $(x,t)\in\Omega\times[0,T]$. Set
  \[\bar s=\sup\big\{0<s\le 1: s({\bds H}^*_{\!i,\ep}, {\bds V}^*_{\!i,\ep})\le(\bds{H}^+_{\!i,\ep}, \bds{V}^+_{\!i,\ep}) \;{\rm ~ in  ~ } \Omega\times[0,T]\big\}.\]
Then $0<\bar s\le1$ and $\bar s({\bds H}^*_{\!i,\ep}, {\bds V}^*_{\!i,\ep})\le(\bds{H}^+_{\!i,\ep}, \bds{V}^+_{\!i,\ep})$ for $(x,t)\in\Omega\times[0,T]$. We shall prove $\bar s=1$. If $\bar s<1$, then $\bds{U}:=\bds{H}^+_{\!i,\ep}-\bar s{\bds H}^*_{\!i,\ep}\ge 0$ and $\bds{Z}:=\bds{V}^+_{\!i,\ep}-\bar s{\bds V}^*_{\!i,\ep}\ge 0$. In view of $\bar s{\bds V}^*_{\!i,\ep}\le\bds{V}^+_{\!i,\ep}<{\bds{V}}+\ep\varphi$ and ${\bds V}^*_{\!i,\ep}>\bar s{\bds V}^*_{\!i,\ep}$, it follows that
 \bess
 &({\bds{V}}+\ep\varphi-\bds{V}^+_{\!i,\ep})^+={\bds{V}}+\ep\varphi-\bds{V}^+_{\!i,\ep},\;\;\;
 (\bds{V}+\ep\varphi-{\bds V}^*_{\!i,\ep})^+
 <\bds{V}+\ep\varphi-\bar s{\bds V}^*_{\!i,\ep},&\\
& (\bds{V}+\ep\varphi-\bds{V}^+_{\!i,\ep})^+-(\bds{V}+\ep\varphi-{\bds V}^*_{\!i,\ep})^+>\bar s{\bds V}^*_{\!i,\ep}-\bds{V}^+_{\!i,\ep}=-\bds{Z}&
 \eess
for $(x,t)\in\oo\times(0,T]$. After careful calculation, it derives that
\bess\begin{cases}
\partial_t\bds{U}-\nabla\cdot d_1(x,t)\nabla\bds{U}=\sigma_1(x,t)H_u(x,t)\bds{Z}-\rho(x,t)\bds{U},\;\;&x\in\Omega,\;0<t\le T,\\
\partial_t\bds{Z}-\nabla\cdot d_2(x,t)\nabla\bds{Z}>\sigma_2(x,t)\big(\bds{V}(x,t)+\ep\varphi(x,t)
-\bds{V}^+_{\!i,\ep}\big)\bds{U}\\[0.1mm]
 \hspace{11mm}-\Big(\mu_1(x,t)+\mu_2(x,t)({\bds{V}}(x,t)-\ep\varphi(x,t))
+\sigma_2(x,t){\bds H}^*_{\!i,\ep}\Big)\bds{Z},\;\;&x\in\Omega,\;0<t\le T,\\[2mm]
a_1\dd\frac{\partial\bds{U}}{\partial\nu}+b_1(x,t)\bds{U}
=a_2\frac{\partial\bds{Z}}{\partial\nu}+b_2(x,t)\bds{Z}=0,\;\;&x\in\partial\Omega,\;0<t\le T,\\[1mm]
\bds{U}(x,0)=\bds{U}(x,T),\ \bds{Z}(x,0)=\bds{Z}(x,T),\;\;&x\in\Omega.
	\end{cases}\eess
In view of $\bds{V}+\ep\varphi-\bds{V}^+_{\!i,\ep}>0$ and $\bds{U}, \bds{Z}\ge 0$. It follows that $\bds{U},\bds{Z}>0$ for $(x,t)\in\Omega\times[0,T]$ by the maximum principle. Then there exists $0<\tau<1-\bar s$ such that $(\bds{U}, \bds{Z})\ge\tau({\bds H}^*_{\!i,\ep}, {\bds V}^*_{\!i,\ep})$, i.e.,
 \bess
 (\bar s+\tau)({\bds H}^*_{\!i,\ep}, {\bds V}^*_{\!i,\ep})\le(\bds{H}^+_{\!i,\ep}, \bds{V}^+_{\!i,\ep}),\;\;\;(x,t)\in\oo\times[0,T].\eess
This contradicts the definition of $\bar s$. Hence $\bar s=1$, i.e., $({\bds H}^*_{\!i,\ep}, {\bds V}^*_{\!i,\ep})\le(\bds{H}^+_{\!i,\ep}, \bds{V}^+_{\!i,\ep})$ for $(x,t)\in\Omega\times[0,T]$. Certainly, ${\bds V}^*_{\!i,\ep}<\bds{V}+\ep\varphi$, and $(\bds{V}+\ep\varphi-{\bds V}^*_{\!i,\ep})^+
=\bds{V}+\ep\varphi-{\bds V}^*_{\!i,\ep}$.

On the other hand, we can find $k>1$ such that $k({\bds H}^*_{\!i,\ep}, {\bds V}^*_{\!i,\ep})\ge(\bds{H}^+_{\!i,\ep}, \bds{V}^+_{\!i,\ep})$ for $(x,t)\in\Omega\times[0,T]$. Set
  \[\ud k=\inf\big\{k\ge1: k({\bds H}^*_{\!i,\ep}, {\bds V}^*_{\!i,\ep})\ge(\bds{H}^+_{\!i,\ep}, \bds{V}^+_{\!i,\ep})\;{\rm ~ in  ~ } \Omega\times[0,T]\big\}.\]
Then $\ud k$ is well defined, $\ud k\ge1$ and $\ud k({\bds H}^*_{\!i,\ep}, {\bds V}^*_{\!i,\ep})\ge(\bds{H}^+_{\!i,\ep}, \bds{V}^+_{\!i,\ep})$ for $(x,t)\in\Omega\times[0,T]$. If $\ud k>1$, then $\bds{P}:=\ud k{\bds H}^*_{\!i,\ep}-\bds{H}^+_{\!i,\ep}\ge 0$ and $\bds{Q}:=\ud k{\bds V}^*_{\!i,\ep}-\bds{V}^+_{\!i,\ep}\ge 0$. Similarly to the above, we can derive $\bds{P}, \bds{Q}>0$ for $(x,t)\in\Omega\times[0,T]$ by the maximum principle, and there exists $0<r<\ud k-1$ such that $(\bds{P}, \bds{Q})\ge r({\bds H}^*_{\!i,\ep}, {\bds V}^*_{\!i,\ep})$, i.e.,
 \[(\ud k-r)({\bds H}^*_{\!i,\ep}, {\bds V}^*_{\!i,\ep})\ge(\bds{H}^+_{\!i,\ep}, \bds{V}^+_{\!i,\ep}),\;\;\;(x,t)\in\oo\times[0,T].\]
This contradicts the definition of $\ud k$. Hence $\ud k=1$ and $({\bds H}^*_{\!i,\ep}, {\bds V}^*_{\!i,\ep})\ge(\bds{H}^+_{\!i,\ep}, \bds{V}^+_{\!i,\ep})$ for $(x,t)\in\Omega\times[0,T]$. The uniqueness of positive solutions of \eqref{3.6} is obtained.

Taking $\ep=0$ in the above, we obtain the  desired conclusion.
\end{proof}\setcounter{equation}{2}

\section{Dynamical properties of \qq{1.1} and \qq{1.2}}\lbl{s3}\setcounter{equation}{0}

\begin{theo}\lbl{th4.1} Let $H_{i0}, V_{u0}, V_{i0}\in C^1(\overline\Omega)$ be positive and satisfy $a_1\frac{\partial H_{i0}}{\partial\nu}+b_1(x,0)H_{i0}=0$ and $a_2\frac{\partial V_{u0}}{\partial\nu}+b_2(x,0)V_{u0}=a_2\frac{\partial V_{i0}}{\partial\nu}+b_2(x,0)V_{i0}=0$ on $\partial\Omega$. Let $(H_i, V_u, V_i)$ be the unique positive solution of \qq{1.1} and \qq{1.2}. Then we have the following conclusions.

{\rm(i)}\; If $\zeta(\mu_1,\beta)<0$ and $\lm(\bds{V})<0$, then
 \bes
\lim_{n\to\yy}(H_i(x,t+nT),\, V_u(x,t+nT),\, V_i(x,t+nT))=(\bds{H}_{\!i}, \bds{V}-\bds{V}_{\!i}, \bds{V}_{\!i})\;\;\; {\rm in}\; \;[C^{2,1}(\overline\Omega\times[0,T])]^3,
\vspace{-2mm}
 \lbl{4.1}\ees
where $\bds{V}$ and $(\bds{H}_{\!i}, \bds{V}_{\!i})$ are the unique positive solutions of \qq{3.3} and \qq{3.4}, respectively, and $\bds{V}_{\!i}<\bds{V}$.

{\rm(ii)}\; If $\zeta(\mu_1,\beta)<0$ and $\lm({\bds{V}})\ge0$, then
 \bes
\lim_{n\to\yy}(H_i(x,t+nT),\, V_u(x,t+nT),\, V_i(x,t+nT))=(0, {\bds{V}}, 0)\;\;\;
{\rm in}\;\; [C^{2,1}(\overline\Omega\times[0,T])]^3.\vspace{-2mm}
 \lbl{4.2}\ees

{\rm(iii)}\; If $\zeta(\mu_1,\beta)\ge 0$, then
 \bess
\lim_{n\to\yy}(H_i(x,t+nT),\, V_u(x,t+nT),\, V_i(x,t+nT))=(0, 0, 0)\;\;\;
{\rm in}\;\; [C^{2,1}(\overline\Omega\times[0,T])]^3.\vspace{-2mm}
 \eess
  \end{theo}

\begin{proof} By use of the regularity theory and compactness argument, it suffices to show that these limits hold uniformly in $\overline\Omega\times[0,T]$.

(i)\, Assume that $\zeta(\mu_1,\beta)<0$ and $\lm({\bds{V}})<0$.

{\it Step 1}. Since $\lm({\bds{V}})<0$, there exists a $0<\ep_0\ll 1$ such that $\lm(\bds{V}; \ep)<0$ when $|\ep|\le\ep_0$.
Let $(H_i, V_u, V_i)$ be the unique solution of \qq{1.1}, and set $V=V_u+V_i$. Then we have that
 \bess
	\begin{cases}
\partial_t V-\nabla\cdot d_2(x,t)\nabla V=\big(\beta(x,t)-\mu_1(x,t)\big)V-\mu_2(x,t)V^2,\;\;&x\in\Omega,\; t>0,\\[1mm]
a_2\dd\frac{\partial V}{\partial\nu}+b_2(x,t)V=0,\;\;&x\in\partial\Omega,\; t>0,\\[1mm]
V(x,0)=V_u(x,0)+V_i(x,0)>0,&\;x\in\Omega,
	\end{cases}
	\eess
and that $\dd\lim_{n\to\yy}V(x,t+nT)={\bds{V}}(x,t)$ in $C^{2,1}(\overline\Omega\times[0,T])$. Let $\varphi(x,t)$ be the positive eigenfunction corresponding to $\zeta(\mu_1,\beta)$ of \eqref{3.2}. Normalize $\varphi$ by $\|\varphi\|_{L^{\infty}(\oo\times(0,T])}=1$. Using analogous methods as in  \cite{Wang24} we can show that for any given $0<\ep\le\ep_0$, there exists a $N\gg 1$ such that
 \bes
0<\bds{V}(x,t)-\ep\varphi(x,t)\le V(x,t+nT)\le \bds{V}(x,t)+\ep\varphi(x,t),\;\;\;(x,t)\in\Omega\times[0,T]\lbl{4.3}
 \ees
for all $n\geq N$. In fact, for the case $a_2=1$, where $a_2$ is given in ${\cal B}_2$, we have $\min_{\overline\Omega\times[0,T]}\varphi(x,t)>0$. Thus, $|V(x,t+nT)-\bds{V}(x,t)|<\ep\varphi(x,t)$ for $(x,t)\in\Omega\times(0,T]$ when $n$ is large. For the case $a_2=0$, we have $\varphi(x,t)>0$ for $(x,t)\in\oo\times(0,T]$, $\varphi(x,t)=0$  for $(x,t)\in\pt\oo\times[0,T]$ and $\partial\varphi(x,t)/\partial\nu<0$ for $(x,t)\in\pt\oo\times[0,T]$, and $(V(x,t+nT)-\bds{V}(x,t))=0$ for $(x,t)\in\pt\oo\times[0,T]$. Since $\dd\lim_{n\to\yy}(V(x,t+nT)-{\bds{V}}(x,t))=0$ in $C^{2,1}(\overline\Omega\times[0,T])$, it is easy to see that there exists $N_1\gg 1$ such that
  \[\frac{\partial\big(V(x,t+nT)-{\bds{V}}(x,t)\big)}{\partial\nu}-\ep\frac{\partial\varphi(x,t)}{\partial\nu}\ge \frac \ep 2\min_{\pt\oo\times[0,T]}\kk|\frac{\partial\varphi(x,t)}{\partial\nu}\rr|>0
  ,\;\;\;(x,t)\in\pt\Omega\times[0,T]\]
for all $n\ge N_1$. This, together with
 \bess
 V(x,t+nT)-\bds{V}(x,t)-\ep\varphi(x,t)=0,\;\;\;(x,t)\in\pt\oo\times[0,T],
 \eess
asserts that there exists $\Omega_0\Subset\Omega$ such that
  \[V(x,t+nT)-\bds{V}(x,t)-\ep\varphi(x,t)<0,\;\;\;(x,t)\in(\Omega\setminus\Omega_0)\times[0,T]\]
for all $n\ge N_1$. Since $\ep\varphi(x,t)>0$ for  $(x,t)\in\overline\Omega_0\times[0,T]$ and $V(x,t+nT)-\bds{V}(x,t)\to 0$ in $C(\overline\Omega_0\times[0,T])$ as $n\to\yy$, there exists $N_2\gg 1$ such that
 \[V(x,t+nT)-\bds{V}(x,t)-\ep\varphi(x,t)<0,\;\;\;(x,t)\in\overline\Omega_0\times[0,T]\]
for all $n\ge N_2$. Thus, $V(x,t+nT)-\bds{V}(x,t)-\ep\varphi(x,t)<0$, i.e., $V(x,t+nT)<\bds{V}(x,t)+\ep\varphi(x,t)$ for $(x,t)\in\Omega\times(0,T]$ and $n\ge\max\{N_1, N_2\}$. Similarly, we can prove that $V(x,t+nT)>\bds{V}(x,t)-\ep\varphi(x,t)$ for $(x,t)\in\Omega\times(0,T]$ when $n$ is large.

Making use of $V_u=V-V_i$ and \qq{4.3}, we see that $(H_i, V_i)$ satisfies
 \bess\begin{cases}
\partial_t H_i-\nabla\cdot d_1(x,t)\nabla H_i=-\rho(x,t)H_i+\sigma_1(x,t)H_u(x,t)V_i,\;\;&x\in\Omega,\;t>NT,\\
\partial_t V_i-\nabla\cdot d_2(x,t)\nabla V_i\leq\sigma_2(x,t)({\bds{V}}(x,t)+\ep\varphi(x,t)-V_i)^+H_i\\
 \hspace{41mm}-\big(\mu_1(x,t)+\mu_2(x,t)({\bds{V}}(x,t)
 -\ep\varphi(x,t))\big)V_i,\;\;&x\in\Omega,\;t>NT,\\[2mm]
a_1\dd\frac{\partial H_i}{\partial\nu}+b_1(x,t)H_i=
a_2\frac{\partial V_i}{\partial\nu}+b_2(x,t)V_i=0,\;\;&x\in\pt\Omega,\;t>NT.
 \end{cases}\eess
Let $(\bds{H}^+_{\!i,\ep}, \bds{V}^+_{\!i,\ep})$ be the unique positive $T$-periodic solution of \eqref{3.6}, and $(\phi^\ep_1,\phi^\ep_2)$ be the positive eigenfunction corresponding to $\lm({\bds{V}}; \ep)$. We can take constants $k\gg 1$ and $0<\delta\ll 1$ such that
 \bess
 k\kk(\bds{H}^+_{\!i,\ep}(x,0), \dd\bds{V}^+_{\!i,\ep}(x,0)\rr) \ge(H_i(x, NT), V_i(x, NT))\ge\delta\kk(\phi^\ep_1(x,0),\dd\phi^\ep_2(x,0)\rr), \;\;\; x\in\Omega.
 \eess
Moreover, it is easy to verify that $k(\bds{H}^+_{\!i,\ep}, \bds{V}^+_{\!i,\ep})$ and $\delta(\phi^\ep_1,\phi^\ep_2)$ are the ordered upper and lower periodic solutions of \eqref{3.6} provided that $k\gg 1$ and $0<\delta\ll 1$. Let $(U_\ep, Z_\ep)$ be the unique positive solution of
 \bes\begin{cases}
\partial_t U_\ep-\nabla\cdot d_1(x,t)\nabla U_\ep=-\rho(x,t)U_\ep+\sigma_1(x,t)H_u(x,t)Z_\ep,\;\;&x\in\Omega,\;t>0,\\
\partial_t Z_\ep-\nabla\cdot d_2(x,t)\nabla Z_\ep=\sigma_2(x,t)({\bds{V}}(x,t)+\ep\varphi(x,t)-Z_\ep)^+U_\ep\\
 \hspace{42mm}-\big(\mu_1(x,t)+\mu_2(x,t)({\bds{V}}(x,t)-
 \ep\varphi(x,t))\big)Z_\ep,\;\;&x\in\Omega,\;t>0,\\[2mm]
a_1\dd\frac{\partial U_\ep}{\partial\nu}+b_1(x,t) U_\ep=a_2\dd\frac{\partial Z_\ep}{\partial\nu}+b_2(x,t)Z_\ep=0,\;\;&x\in\pt\Omega,\;t>0,\\[1mm]
(U_\ep(x,0), Z_\ep(x,0))=k(\bds{H}^+_{\!i,\ep}(x,0), \bds{V}^+_{\!i,\ep}(x,0)),&x\in\Omega.
 \end{cases}\lbl{4.4}\ees
By the comparison principle, $(U_\ep(x,t), Z_{\ep}(x,t))\ge\delta\kk(\phi^\ep_1(x,t),\dd\phi^\ep_2(x,t)\rr)$, and
\bes
(H_i(x,t+NT),V_i(x,t+NT))\leq (U_\ep(x,t), Z_{\ep}(x,t))\leq (k\bds{H}^+_{\!i,\ep}(x,t),k\bds{V}^+_{\!i,\ep}(x,t))\lbl{4.5}\ees
for $x\in\Omega$ and $t>0$. Define $U_\ep^n(x,t)=U_\ep(x,t+nT)$ and $Z_\ep^n(x,t)=Z_\ep(x,t+nT)$ for $(x,t)\in\Omega\times[0,T]$. Then
 \bes\begin{cases}
 U_\ep^n(x,t)\ge\delta\phi^\ep_1(x,t+nT)
 =\delta\phi^\ep_1(x,t),\;\;\;(x,t)\in\Omega\times(0,T],\\
  Z_\ep^n(x,t)\ge\delta\phi^\ep_2(x,t+nT)=
  \delta\phi^\ep_2(x,t),\;\;\;(x,t)\in\Omega\times(0,T].
  \end{cases}\lbl{4.6}\ees
On the other hand, since  $d_i(x,t)$, $\rho(x,t)$, $\sigma_i(x,t)$, ${\bds{V}}(x,t)$, $\varphi(x,t)$ and $\mu_i(x,t)$ $(i=1,2)$ are time periodic functions with periodic $T$, we have
\bess\begin{cases}
\partial_t U_\ep^n-\nabla\cdot d_1(x,t)\nabla U_\ep^n=-\rho(x,t)U_\ep^n+\sigma_1(x,t) H_u(x,t)Z_\ep^n,\;\;&x\in\Omega,\;0<t\le T,\\
\partial_t Z_\ep^n-\nabla\cdot d_2(x,t)\nabla Z_\ep^n=\sigma_2(x,t)(\bds{V}(x,t)+\ep\varphi(x,t)-Z_\ep^n)^+U_\ep^n\\
 \hspace{43mm}-\kk[\mu_1(x,t)+\mu_2(x,t)(\bds{V}(x,t)-\ep\varphi(x,t))\rr]Z_\ep^n,\;\;\;&x\in\Omega,\;0<t\le T,\\[2mm]
a_1\dd\frac{\partial U_\ep^n}{\partial\nu}+b_1(x,t)U_\ep^n
=a_2\frac{\partial Z_\ep^n}{\partial\nu}+b_2(x,t)Z_\ep^n=0,\;\;&x\in\partial\Omega,\;0<t\le T,\\[1mm]
U_\ep^n(x,0)=U_\ep^{n-1}(x,T), \ \ Z_\ep^n(x,0)=Z_\ep^{n-1}(x,T),\;\;&x\in\Omega.
\end{cases}\eess
Note that
\bess
&U_\ep^1(x,0)=U_\ep(x,T)\leq k\bds{H}^+_{\!i,\ep}(x,T)
=k\bds{H}^+_{\!i,\ep}(x,0)=U_\ep(x,0),\;\;\;x\in\Omega,\\
&Z_\ep^1(x,0)=Z_\ep(x,T)\leq k\bds{V}^+_{\!i,\ep}(x,T)
=k\bds{V}^+_{\!i,\ep}(x,0)=Z_\ep(x,0),\;\;\;x\in\Omega.
\eess
By the comparison principle,
\bess
(U_\ep^1(x,t), Z_\ep^1(x,t))\leq (U_\ep(x,t), Z_\ep(x,t)),\;\;\;(x,t)\in\Omega\times(0,T],
\eess
and then
  \[U_\ep^2(x,0)=U_\ep^1(x,T)\leq U_\ep(x,T)=U_\ep^1(x,0),\;\;\; Z_\ep^2(x,0)=Z_\ep^1(x,T)\leq Z_\ep(x,T)=Z_\ep^1(x,0)\]
for $x\in\Omega$. It is derived that
 \[(U_\ep^2(x,t), Z_\ep^2(x,t))\leq (U_\ep^1(x,t),
  Z_\ep^1(x,t)),\;\;(x,t)\in\Omega\times(0,T]\]
by the comparison principle. Utilizing the inductive method, we can show that $U_\ep^n$ and $Z_\ep^n$ are monotonically decreasing in $n$.  This combined with \qq{4.6} indicates that there exists a function pair $(\bds{U}_{\!\!\ep}, \bds{Z}_\ep)$ satisfying
 \[(\bds{U}_{\!\!\ep}(x,t), \bds{Z}_\ep(x,t))\ge\delta
 \kk(\phi^\ep_1(x,t),\dd\phi^\ep_2(x,t)\rr),\;\;\;(x,t)\in\Omega\times(0,T]\]
such that $(U_\ep^n, Z_\ep^n)\to (\bds{U}_{\!\!\ep}, \bds{Z}_\ep)$ pointwisely in $\overline\Omega\times[0,T]$ as $n\to\infty$. Clearly, $\bds{U}_{\!\!\ep}(\cdot,0)=\bds{U}_{\!\!\ep}(\cdot, T)$ and $\bds{Z}_\ep(\cdot,0)=\bds{Z}_\ep(\cdot, T)$. By use of the regularity theory and compactness argument, it can be proved that $(U_\ep^n, Z_\ep^n)\to (\bds{U}_{\!\!\ep}, \bds{Z}_\ep)$ in $[C^{2,1}(\overline\Omega\times[0,T])]^2$ as $n\to\infty$. Clearly, $(\bds{U}_{\!\!\ep}, \bds{Z}_\ep)$ satisfies \eqref{3.6}, i.e., $(\bds{U}_{\!\!\ep}, \bds{Z}_\ep)$ is a positive solution of \eqref{3.6}. Thus, $(\bds{U}_{\!\!\ep}, \bds{Z}_\ep)=(\bds{H}^+_{\!i,\ep}, \bds{V}^+_{\!i,\ep})$ by the uniqueness of positive solution of \eqref{3.6}. Consequently,
\bess
&\dd\lim_{n\to\infty}U_\ep(x,t+nT)=\bds{H}^+_{\!i,\ep}(x,t),\;\;\;
&\dd\lim_{n\to\infty}Z_\ep(x,t+nT)=\bds{V}^+_{\!i,\ep}(x,t) \;\;\; {\rm in}\;\; C^{2,1}(\overline\Omega\times[0,T]).
\eess
Note that $\dd\lim_{\ep\to0}({\bds{H}^+_{\!i,\ep}}, {\bds{V}^+_{\!i,\ep}})=({\bds{H}_{i}}, {\bds{V}_{\!i}})$ in $[C^{2,1}(\overline\Omega\times[0,T])]^2$, it follows from \eqref{4.5} that
\bes
\dd\limsup_{n\to\infty}(H_i(x,t+nT), V_i(x,t+nT))\leq (\bds{H}_{\!i}(x,t), \bds{V}_{\!i}(x,t)) \;\;\;{\rm uniformly\; in}\; \;\overline\Omega\times[0,T].
\lbl{4.7}\ees

{\it Step 2}. Let $\bds{W}(x,t)=\bds{V}(x,t)-\bds{V}_{\!i}(x,t)$. Then $\bds{W}(x,t)>0$ for $(x,t)\in\Omega\times(0,T]$ (Theorem \ref{th3.1}). The direct calculation yields that $\bds{W}$ satisfies
 \bess\begin{cases}
 \partial_t\bds{W}-\nabla\cdot d_2(x,t)\nabla\bds{W}+\big(\mu_1(x,t)+\mu_2(x,t)\bds{V}(x,t)
 +\sigma_2(x,t)\bds{H}_{\!i}(x,t)\big)\bds{W}\\
\hspace{40mm}=\beta(x,t)\bds{V}(x,t),\;&x\in\Omega,\;0<t\le T,\\[2mm]
a_2\dd\frac{\partial\bds{W}}{\partial\nu}+b_2(x,t)\bds{W}=0,\;\;&x\in\partial\Omega,\;0<t\le T,\\[1mm]
 \bds{W}(x,0)=\bds{W}(x,T),\;\;&x\in\Omega.
 \end{cases}\eess
Since $\mu_1(x,t)+\mu_2(x,t)\bds{V}(x,t)+\sigma_2(x,t)\bds{H}_{\!i}(x,t)>0$ and $\beta(x,t)\bds{V}(x,t)\ge,\,\not\equiv 0$ for $(x,t)\in\Omega\times(0,T]$, it follows from the maximum principle (\cite[Theorem 7.1]{Wpara}) that $\bds{W}(x,t)>0$ for $(x,t)\in\overline\Omega\times[0,T]$ when $a_2=1$, and $\bds{W}(x,t)>0$ for $(x,t)\in\Omega\times(0,T]$ when $a_2=0$. Then we have the following conclusion:
\vspace{-2mm}
\begin{enumerate}[leftmargin=16mm]
\item[{\bf Claim}:]  there exists a $0<\tau_0<\ep_0$ such that, for all $0<\tau<\tau_0$,
  \bess
 \bds{W}(x,t)>2\tau\varphi(x,t), \;\;{\rm i.e.,}\;\; \bds{V}_{\!i}(x,t)<\bds{V}(x,t)-2\tau\varphi(x,t),\;\;(x,t)\in\Omega\times[0,T].
  \eess
\end{enumerate}
In fact, if $\bds{W}(x,t)>0$ in $\overline\Omega\times[0,T]$, this claim is obvious. If $\bds{W}(x,t)>0$ for $(x,t)\in\Omega\times(0,T]$, then $\partial\bds{W}(x,t)/\partial\nu<0$ on $\pt\oo\times[0,T]$ by the Hopf boundary lemma, and this claim can be proved by adapting analogous methods as in \cite[Lemma 2.1]{WPbook24}.

For such $\tau$, since $\dd\lim_{\ep\to0}\bds{Z}_\ep=\lim_{\ep\to0}\bds{V}^+_{\!i,\ep}=\bds{V}_{\!i}$ in $C^{2,1}(\overline\Omega\times[0,T])$, by the analogous discussions of \qq{4.3}, we can find $0<\bar\ep<\ep_0$ such that
 \bess
 \bds{Z}_{\bar\ep}(x,t)<\bds{V}_{\!i}(x,t)+\frac\tau2\vp(x,t),
 \;\;\;(x,t)\in\Omega\times[0,T].\eess
Recalling that $\dd\lim_{n\to\yy}Z_{\bar\ep}(x,t+nT)=\bds{Z}_{\bar\ep}(x,t)$ in $C^{2,1}(\overline\Omega\times[0,T])$ and using \qq{4.3}, we can find $\tilde N\gg 1$ such that
 \bes
 Z_{\bar\ep}(x,t+nT)<\bds{Z}_{\bar\ep}(x,t)+\frac\tau2\vp(x,t)<\bds{V}_{\!i}(x,t)+\tau \vp(x,t),\;\;\;(x,t)\in\Omega\times[0,T]
 \lbl{4.8}\ees
for all $n\geq \tilde N$. By \qq{4.5}, $V_i(x,t+(n+N)T)\leq Z_{\bar\ep}(x,t+nT)$ holds for $(x,t)\in\Omega\times(0,T]$. Thanks to the above claim and \qq{4.8}, it follows that
\bess
 V_i(x,t+(n+N)T)<\bds{V}_{\!i}(x,t)+\tau \vp(x,t)<\bds{V}(x,t)-\tau\varphi(x,t),\;\;\;(x,t)\in\Omega\times[0,T]
 \eess
for all $n\geq \tilde N$. Set $N_*=\tilde N+N+1$. Since $V_u=V-V_i$, we have that $(H_i, V_i)$ satisfies
 \bess\begin{cases}
\partial_t H_i-\nabla\cdot d_1(x,t)\nabla H_i=-\rho(x,t)H_i+\sigma_1(x,t) H_u(x,t)V_i,\;\; &x\in\oo,\;t>N_*T,\\
\partial_t V_i-\nabla\cdot d_2(x,t)\nabla V_i\geq\sigma_2(x,t)({\bds{V}}(x,t)-\tau\varphi(x,t)-V_i)^+H_i\\
 \hspace{42mm}-\big(\mu_1(x,t)+\mu_2(x,t)({\bds{V}}(x,t)
 +\tau\varphi(x,t))\big)V_i,\;\; &x\in\oo,\;t>N_*T,\\[2mm]
a_1\dd\frac{\partial H_i}{\partial\nu}+b_1(x,t)H_i=a_2\frac{\partial V_i}{\partial\nu}+b_2(x,t)V_i=0,\;\; &x\in\pt\oo,\;t>N_*T.
 \end{cases}\eess

Let $(\phi^{-\tau}_1,\phi^{-\tau}_2)$ be the positive eigenfunction corresponding to $\lm({\bds{V}}; -\tau)$. Similar to Step 1, we can find $0<\delta\ll 1$ such that $\delta(\phi^{-\tau}_1,\phi^{-\tau}_2)$ is a lower solution of \eqref{3.6} with $\ep=-\tau$, and
 \[\delta(\phi^{-\tau}_1(x,0),\phi^{-\tau}_2(x,0))\le (H_i(x,NT), V_i(x,NT)),\;\;\; x\in\Omega.\]
Let $(P_\tau, R_\tau)$ be the unique positive solution of
 \bess\begin{cases}
\partial_t P_\tau-\nabla\cdot d_1(x,t)\nabla P_\tau=-\rho(x,t)P_\tau+\sigma_1(x,t) H_u(x,t)R_\tau,
\;\; &x\in\oo,\;t>0,\\
\partial_t R_\tau-\nabla\cdot d_2(x,t)\nabla R_\tau=\sigma_2(x,t)({\bds{V}}(x,t)-\tau\varphi(x,t)-R_\tau)^+P_\tau\\
 \hspace{44mm}-\big(\mu_1(x,t)+\mu_2(x,t)({\bds{V}}(x,t)
 +\tau\varphi(x,t))\big)R_\tau,\;\; &x\in\oo,\;t>0,\\[2mm]
a_1\dd\frac{\partial P_\tau}{\partial\nu}+b_1(x,t)P_\tau=a_2\frac{\partial R_\tau}{\partial\nu}+b_2(x,t)R_\tau=0,\;\; &x\in\partial\oo,\;t>0,\\[1mm]
(P_\tau(x,0), R_\tau(x,0))=\delta(\phi^{-\tau}_1(x,0),\phi^{-\tau}_2(x,0)),\;\;&x\in\Omega.
\end{cases}\eess
By the comparison principle,
 \bes
(H_i(x,t+N_*T), V_i(x,t+N_*T))\geq (P_\tau(x,t), R_\tau(x,t))\geq (\delta \phi^{-\tau}_1(x,t),\delta \phi^{-\tau}_2(x,t))
\lbl{4.9}\ees
for $x\in\oo$ and $t>0$. Set $P_\tau^n(x,t)=P_\tau(x,t+nT)$ and $R_\tau^n(x,t)=R_\tau(x,t+nT)$ for $(x,t)\in\Omega\times[0,T]$. Then $(P_\tau^n, R_\tau^n)$ satisfies
 \bess\begin{cases}
\partial_t P_\tau^n-\nabla\cdot d_1(x,t)\nabla P_\tau^n=-\rho(x,t)P_\tau^n+\sigma_1(x,t) H_u(x,t)R_\tau^n,\;\;&x\in\Omega,\;0<t\le T,\\
\partial_t R_\tau^n-\nabla\cdot d_2(x,t)\nabla R_\tau^n=\sigma_2(x,t)({\bds{V}}(x,t)
-\tau\varphi(x,t)-R_\tau^n)^+P_\tau^n\\
 \hspace{44mm}-\big(\mu_1(x,t)+\mu_2(x,t)({\bds{V}}(x,t)+\tau\varphi(x,t))\big)R_\tau^n,\;\;&x\in\Omega,\;0<t\le T,\\[2mm]
a_1\dd\frac{\partial P_\tau^n}{\partial\nu}+b_1(x,t)P_\tau^n
=a_2\frac{\partial R_\tau^n}{\partial\nu}+b_2(x,t)R_\tau^n=0,
\;\;&x\in\partial\Omega,\;0<t\le T,\\[1mm]
P_\tau^n(x,0)=P_\tau^{n-1}(x,T), \;\;R_\tau^n(x,0)=R_\tau^{n-1}(x,T),&x\in\Omega.
\end{cases}\eess
Similar to Step 1, $P_\tau^n$ and $R_\tau^n$ are monotonically increasing in $n$, and
 \bess
\lim_{n\to\infty}P_\tau(x,t+nT)=\bds{H}^+_{\!i,-\tau}(x,t), \ \ \
\lim_{n\to\infty}R_\tau(x,t+nT)=\bds{V}^+_{\!i,-\tau}(x,t) \ \ {\rm in}\;\; C^{2,1}(\overline\Omega\times[0,T]),
 \eess
where $(\bds{H}^+_{\!i,-\tau}, \bds{V}^+_{\!i,-\tau})$ is the unique positive periodic solution of problem \eqref{3.6} with $\ep=-\tau$. Take advantage of  \eqref{4.9} and $\dd\lim_{\tau\to 0}(\bds{H}^+_{\!i,-\tau}, \bds{V}^+_{\!i,-\tau})=(\bds{H}_{i}, \bds{V}_{\!i})$, it follows that
\bess
\liminf_{n\to\infty}(H_i(x,t+nT), V_i(x,t+nT))\geq (\bds{H}_{\!i}(x,t), \bds{V}_{\!i}(x,t))\;\;\;{\rm uniformly\; in}\;\;\overline\Omega\times[0,T].
\eess
This, together with \eqref{4.7}, yields that
\bess
\lim_{n\to\infty}(H_i(x,t+nT),V_i(x,t+nT))=(\bds{H}_{\!i}(x,t),\bds{V}_{\!i}(x,t)) \;\;\;{\rm uniformly\; in}\;\;\overline\Omega\times[0,T].
\eess

Using the fact that $\dd\lim_{n\to\yy}(V_u(x,t+nT)+V_i(x,t+nT))={\bds{V}}(x,t)$ in $C^{2,1}(\overline\Omega\times[0,T])$ and the uniform estimate (cf. \cite[Theorems 2.11 and 3.14]{Wpara}), we see that \qq{4.1} holds.

(ii)\, Assume that $\zeta(\mu_1,\beta)<0$ and $\lm({\bds{V}})\ge0$. Let $\lm({\bds{V}}; \ep)$ be the principal eigenvalue of \eqref{3.7}. If $\lm({\bds{V}})>0$, then $\lm({\bds{V}}; \ep)>0$ when $0<\ep\ll 1$. From the proof of necessity of Theorem \ref{th3.1} we see that \eqref{3.6a} has no positive solution. If \eqref{3.6} has a positive solution $(\bds{H}_{\!i,\ep}^+, \bds{V}_{\!i, \ep}^+)$, then $(\bds{H}_{\!i,\ep}^+, \bds{V}_{\!i, \ep}^+)< (\ol{\bds{H}}_{\!i}, \bds{V}+\ep\varphi)$ by the comparison principle since  $(\ol{\bds{H}}_{\!i}, \bds{V}+\ep\varphi)$ is a strict upper solution of \eqref{3.6}, and so  $(\bds{H}_{\!i,\ep}^+, \bds{V}_{\!i, \ep}^+)$ is a positive solution of \eqref{3.6a}. This is impossible.

Let $(U,Z)$ be the unique positive solution of \qq{4.4} with initial data $(C, K)$, where $C$ and $K$ are suitably large positive constants, for example,
  \bess
 K&>&\max_{\overline\Omega\times[0,T]}\max\{{\bds{V}}(x,t)+\ep\varphi(x,t),\,V_i(x,NT)\},\\
 C&>&\max_{\overline\Omega\times[0,T]}\max\kk\{K\frac{\sigma_1(x,t)}{\rho(x,t)}H_u(x,t),\, H_i(x,NT)\rr\},\eess
such that $(C, K)$ is an upper solution of \qq{3.6}. Similar to the above,  \bess
 \dd\lim_{n\to\yy}(U(x,t+nT),Z(x,t+nT))=(\bds{H^+_{i,\ep}}(x,t), \bds{V^+_{i,\ep}}(x,t))=(0,0)\;\;\;{\rm in}\;\;[C^{2,1}(\overline\Omega\times[0,T])]^2
  \eess
since \eqref{3.6} has no positive solution. By the comparison principle we have $\dd\lim_{t\to\yy}(H_i, V_i)=(0,0)$. This, together with
$\dd\lim_{n\to\yy}(V_u(x,t+nT)+V_i(x,t+nT))={\bds{V}}(x,t)$,
yields the limit \qq{4.2}.

If $\lm({\bds{V}})=0$, then $\lm({\bds{V}}; \ep)<0$ and \eqref{3.6} has a unique positive solution $(\bds{H^+_{i,\ep}}, \bds{V^+_{i,\ep}})$ for any $\ep>0$. Obviously, $\dd\lim_{\ep\to 0}(\bds{H^+_{i,\ep}}, \bds{V^+_{i,\ep}})=(0, 0)$ in $[C^{2,1}(\overline\Omega\times[0,T])]^2$. Similar to the above, $\dd\lim_{t\to\yy}(H_i, V_i)=(0,0)$ and \qq{4.2} holds.

(iii)\; If $\zeta(\mu_1,\beta)\ge 0$, then \eqref{3.3} has no positive solution. Therefore, $\dd\lim_{n\to\yy}V(x,t+nT)=0$, and then
   \bess
 \lim_{n\to\yy}V_u(x,t+nT)=\dd\lim_{n\to\yy}V_i(x,t+nT))=0\eess
uniformly in $\overline\Omega\times[0,T]$. So, $\dd\lim_{n\to\yy}H_i(x,t+nT)=0$  uniformly in $\overline\Omega\times[0,T]$ by the first equation of \qq{1.1}.
The proof is complete.
 \end{proof}

\vskip 4pt \noindent {\bf Funding} The first author was supported by National Natural Science Foundation of China (No. 12171120). The second author was supported by National Natural Science Foundation of China (No. 12201457, 12271401).\vskip 4pt

\vskip 4pt \noindent {\bf Declarations} The authors have no relevant financial or non-financial interests to disclose.

\end{document}